\documentclass[american,a4paper,12pt]{amsart}
\pdfoutput=1
\usepackage[utf8]{inputenc}
\usepackage{mathrsfs}
\usepackage{enumitem}
\usepackage{mathtools}
\usepackage{amssymb}
\usepackage{graphicx}
\usepackage[unicode=true,pdfusetitle,
 bookmarks=true,bookmarksnumbered=false,bookmarksopen=false,
 breaklinks=false,pdfborder={0 0 1},backref=false,colorlinks=false]{hyperref}
\usepackage{breakurl}
\usepackage{color}
\usepackage[foot]{amsaddr}
\usepackage{cite}
\usepackage[left=80pt, right=80pt]{geometry}

\newtheorem{thm}{Theorem}[section]

\newtheorem{cor}[thm]{Corollary}

\theoremstyle{definition}

\theoremstyle{remark}
\newtheorem{rem}{Remark}[section]

\begin{document}
\title[Traveling waves for strongly coupled competitive systems]{Spatial segregation limit of traveling wave solutions for a fully nonlinear strongly coupled competitive system}
\author{Léo Girardin}
\email{leo.girardin@math.cnrs.fr}
\address[L.G.]{Institut Camille Jordan, Universit\'{e} Claude Bernard Lyon-1, 43 boulevard du 11 novembre 1918, 69622 Villeurbanne Cedex, France}
\thanks{This work was supported by a public grant as part of the Investissement d'avenir project,
reference ANR-11-LABX-0056-LMH, LabEx LMH}

\author{Danielle Hilhorst}
\email{danielle.hilhorst@math.u-psud.fr}
\address[D.H.]{Universit\'{e} Paris-Saclay, CNRS, Laboratoire de Math\'{e}matiques d'Orsay, 91405 Orsay Cedex, France}

\begin{abstract}
    The paper is concerned with a singular limit for the bistable traveling wave problem in a very large class 
    of two-species fully nonlinear parabolic systems with competitive reaction terms. Assuming existence of traveling
    waves and enough compactness, we derive and characterize the limiting problem. The assumptions and results
    are discussed in detail. The free boundary problem obtained at the limit is specified for important applications.
\end{abstract}


\maketitle
\section{Introduction}

We investigate the reaction--diffusion system
\begin{equation}\label{eq:RD_system}
    \begin{cases}
	\partial_t u = & \partial_{x}\left( d_{1,1}\left( u,v \right)\partial_{x} u+d_{1,2}\left( u,v \right)\partial_{x} v \right)\\
	& + h_{1,1}\left( u,v \right)\partial_{x} u+h_{1,2}\left( u,v \right)\partial_{x} v \\
	& +ug_1(u,v)-k\omega(u,v), \\
	\partial_t v = & \partial_{x}\left( d_{2,1}\left( u,v \right)\partial_{x} u+d_{2,2}\left( u,v \right)\partial_{x} v \right)\\
	&+ h_{2,1}\left( u,v \right)\partial_{x} u+h_{2,2}\left( u,v \right)\partial_{x} v \\
	& +vg_2(u,v)-\alpha k\omega(u,v),
    \end{cases}
\end{equation}
where $t$ is a real time variable, $x$ is a one-dimensional real space variable, $u(t,x)$ and $v(t,x)$ are two population densities,
$\mathbf{D}=(d_{i,j})_{1\leq i,j\leq 2}$ is a matrix of self- and cross-diffusion rates \footnote{Although the word ``rate'' might
be misleading, in the whole paper, rates are not necessarily constant and are in general functions of $(u,v)$.},
$\mathbf{H}=(h_{i,j})_{1\leq i,j\leq 2}$ is a matrix of self- and cross-advection rates,
$\mathbf{g}=(g_i)_{1\leq i\leq 2}$ is a vector of growth rates per capita accounting for a saturation effect, 
$\omega$ is an interpopulation competition operator and
$k>0$ and $\alpha k>0$ are constant interpopulation competition rates exerted respectively
by $v$ on $u$ and by $u$ on $v$.

The three main examples of such systems that we have in mind are the classical Lotka--Volterra
competition--diffusion system,
\begin{equation}\label{eq:LV_system}
    \begin{cases}
	\partial_t u = \partial_{xx} u +u(1-u)-kuv, \\
	\partial_t v = d\partial_{xx} v+rv(1-v)-\alpha kuv,
    \end{cases}
\end{equation}
the Potts--Petrovskii competition--diffusion--cross-taxis system \cite{Potts_Petrovskii},
\begin{equation}\label{eq:PP_system}
    \begin{cases}
	\partial_t u = \partial_{xx}u -\gamma_1\partial_{x}\left( u\partial_x v \right) + u(1-u) -kuv, \\
	\partial_t v = d\partial_{xx}v -\gamma_2\partial_{x}\left( v\partial_x u \right) + rv(1-v) -\alpha kuv,
    \end{cases}
\end{equation}
and the Shigesada--Kawasaki--Teramoto competition--cross-diffusion system \cite{SKT_79},
\begin{equation}\label{eq:SKT_system}
    \begin{cases}
	\partial_t u = \partial_{xx}\left( u(d_1+a_{1,1}u+a_{1,2}v) \right) + u(1-u) -kuv, \\
	\partial_t v = \partial_{xx}\left( v(d_2+a_{2,1}u+a_{2,2}v) \right) + rv(1-v) -\alpha kuv.
    \end{cases}
\end{equation}
In the above systems, the various parameters are all positive constants, except $\gamma_1$ and 
$\gamma_2$ that might be of any sign.

Our abstract framework is justified by the fact that we want to deal with all these
systems simultaneously. We also want to encompass generalizations of \eqref{eq:SKT_system}
of the form
\begin{equation*}
    \begin{cases}
	\partial_t u = \partial_{xx}\left( u(d_1+a_{1,1}u^{\beta_{1,1}}+a_{1,2}v^{\beta_{1,2}}) \right) + u(1-u) -kuv, \\
	\partial_t v = \partial_{xx}\left( v(d_2+a_{2,1}u^{\beta_{2,1}}+a_{2,2}v^{\beta_{2,2}}) \right) + rv(1-v) -\alpha kuv,
    \end{cases}
\end{equation*}
with arbitrary positive exponents $\left( \beta_{i,j} \right)_{1\leq i,j\leq 2}$,
that have been considered in the literature \cite{Desvillettes_et_al_2015}.
Moreover, we want to be able to replace the logistic growth terms $u(1-u)$ and 
$rv(1-v)$ above by more
general monostable reaction terms, for instance with a weak Allee effect (\textit{e.g.}, 
$u(u+\theta)(1-u)$ with $\theta\in[0,1)$).

More precisely, we are interested in the singular limit $k\to+\infty$ of the associated traveling wave system
satisfied by entire solutions of the form $(u,v):(t,x)\mapsto\left( \phi,\psi \right)(x-ct)$, where $(\phi,\psi)$ is
the wave profile and $c$ is the wave speed. From now on, we denote $\xi=x-ct$ the wave variable and
$\boldsymbol{\Phi}=(\phi,\psi)$ the wave profile vector. The traveling wave system reads, in vector form,
\begin{equation}\label{eq:TW_system}
    -(\mathbf{D}(\boldsymbol{\Phi})\boldsymbol{\Phi}')'-(\mathbf{H}(\boldsymbol{\Phi})+c\mathbf{I})\boldsymbol{\Phi}' = 
    \boldsymbol{\Phi}\circ\mathbf{g}(\boldsymbol{\Phi}) - k\omega(\boldsymbol{\Phi})
    \begin{pmatrix}
	1 \\
	\alpha
    \end{pmatrix}.
\end{equation}
The notation $\circ$ above stands for the component-by-component product of two vectors, namely 
the so-called Hadamard product.

The singular limit $k\to+\infty$ of the traveling wave system for \eqref{eq:LV_system} was studied by the first author 
and G. Nadin in \cite{Girardin_Nadin_2015}. There, it was proved that the singular limit of 
the solution $(\phi_k,\psi_k,c_k)$ of
\begin{equation}\label{eq:TW_LV_system}
    \begin{cases}
	-\phi_k''-c_k\phi_k' =\phi_k(1-\phi_k)-k\phi_k\psi_k, \\
	-d\psi_k''-c_k\psi_k'=r\psi_k(1-\psi_k)-\alpha k\phi_k\psi_k,\\
	\lim_{-\infty}\boldsymbol{\Phi}_k=(1,0), \\
	\lim_{+\infty}\boldsymbol{\Phi}_k=(0,1), \\
	\phi_k'<0, \\
	\psi_k'>0,
    \end{cases}
\end{equation}
is 
\begin{equation*}
    (\phi_\infty,\psi_\infty,c_\infty)=\left(\alpha^{-1}z^+,z^-,c\right),
\end{equation*}
where $z^+=\max(z,0)$ and $z^-=\max(-z,0)$ are respectively the positive and negative part
of the function $z$ and where $(z,c)$ is the unique (up to shifts of $z$) solution of
\begin{equation*}
    \begin{cases}
	\int-((\mathbf{1}_{z>0}+d\mathbf{1}_{z<0})z')'\chi-c\int z'\chi=\int(z^+(1-\frac{z}{\alpha})-rz^-(1+z))\chi\quad\text{for all }\chi\in L^1(\mathbb{R}),\\
	\lim_{-\infty}z=\alpha, \\
	\lim_{+\infty}z=-1,\\
	z'\leq 0.
    \end{cases}
\end{equation*}
The integral equality is the equation 
$-((\mathbf{1}_{z>0}+d\mathbf{1}_{z<0})z')'-cz'=z^+(1-\frac{z}{\alpha})-rz^-(1+z)$
tested against test functions in $L^1(\mathbb{R})$ and where the divergence-form second order 
term, the first order term and the zeroth order term are all in $L^\infty(\mathbb{R})$. Consequently,
$z$ is smooth and a classical solution away from the point $z^{-1}(0)$ and Lipschitz-continuous 
at $z^{-1}(0)$.

Roughly speaking, the preceding result follows from the combination of $\phi_k\psi_k\to0$ 
(spatial segregation) and an equation that does not depend explicitly on $k$:
\begin{equation*}
	-\alpha\phi_k''+d\psi_k''-c_k\left( \alpha\phi_k'-\psi_k' \right)=\alpha\phi_k(1-\phi_k)-r\psi_k\left( 1-\psi_k \right).
\end{equation*}

In the present work, we will show that the singular limit $k\to+\infty$ of \eqref{eq:TW_system} is 
similarly given, under reasonable assumptions, by the unique (up to shifts) solution of
\begin{equation}\label{eq:singular_limit}
    \begin{cases}
	\int-(d(z)z')'\chi-\int(c+h(z))z'\chi=\int g(z)\chi\quad\text{for all }\chi\in L^1(\mathbb{R}),\\
	\lim_{-\infty}z=\alpha, \\
	\lim_{+\infty}z=-1,\\
	z'\leq 0.
    \end{cases}
\end{equation}
where 
\begin{equation}\label{def:d}
    d:z\mapsto d_{1,1}\left( \frac{z}{\alpha},0 \right)\mathbf{1}_{z>0}+d_{2,2}\left( 0,-z \right)\mathbf{1}_{z<0},   
\end{equation}
\begin{equation}\label{def:h}
    h:z\mapsto h_{1,1}\left( \frac{z}{\alpha},0 \right)\mathbf{1}_{z>0}+h_{2,2}\left( 0,-z \right)\mathbf{1}_{z<0},
\end{equation}
\begin{equation}\label{def:g}
    g:z\mapsto z^+ g_1\left(\frac{z}{\alpha},0\right)-z^- g_2(0,-z).
\end{equation}

\subsection{Minimal speeds of the underlying monostable problems}
In order to state more precisely the assumptions and the results, we need to introduce
$c_\star^+\in\mathbb{R}$ the minimal speed of classical traveling wave solutions of
\begin{equation*}
    \begin{cases}
	-(d(z)z')'-(c+h(z))z'=g(z),\\
	\lim_{-\infty}z=\alpha, \\
	\lim_{+\infty}z=0,\\
	z'\leq 0,
    \end{cases}
\end{equation*}
and $c_\star^-\in\mathbb{R}$ the minimal speed of classical traveling wave solutions of
\begin{equation*}
    \begin{cases}
	-(d(z)z')'-(c+h(z))z'=g(z),\\
	\lim_{-\infty}z=-1, \\
	\lim_{+\infty}z=0,\\
	z'\geq 0.
    \end{cases}
\end{equation*}
Both are well-defined provided $d$, $h$ and $g$ are continuous away from $0$ 
and $d$ is positive away from $0$ \cite{Malaguti_Marcelli_2002}.

\subsection{Standing assumptions}
The standing assumptions follow. 

\begin{enumerate}[label=$({\mathsf{A}}_{\arabic*})$]
    \item \label{ass:regularity} $\mathbf{D}\in \mathscr{C}^1([0,1]^2)$, $\mathbf{H}\in \mathscr{C}([0,1]^2)$, $\mathbf{g}\in \mathscr{C}([0,1]^2)$,
	$\omega\in\mathscr{C}([0,1]^2)$.
    \item \label{ass:ellipticity} $d_{1,1}$ and $d_{2,2}$ are positive in $[0,1]^2$.
    \item \label{ass:no_coupling_at_zero} $d_{1,2}(0,\bullet)$, $d_{2,1}(\bullet,0)$, $h_{1,2}(0,\bullet)$ and $h_{2,1}(\bullet,0)$ are identically zero in $[0,1]$.
    \item \label{ass:monostability} $g_1(\bullet,0)$ and $g_2(0,\bullet)$ are positive in $(0,1)$ and $g_1(1,0)=g_2(0,1)=0$.
    \item \label{ass:symmetric_competition} $\omega$ is positive in $(0,1]^2$ and identically zero on $\{0\}\times[0,1]\cup[0,1]\times\{0\}$.
    \item \label{ass:ordered_monostable_minimal_speeds} $c_\star^+>-c_\star^-$.
    \item \label{ass:existence_waves} There exists $k^\star>0$ such that, for all $k>k^\star$, 
	the system \eqref{eq:TW_system} admits a classical solution $(\boldsymbol{\Phi}_k,c_k)$ with nonincreasing
	$\phi_k$, nondecreasing $\psi_k$, $\lim_{-\infty}\boldsymbol{\Phi}_k=(1,0)$, $\lim_{+\infty}\boldsymbol{\Phi}_k=(0,1)$.
    \item \label{ass:bounded_speeds} Any family of wave speeds $(c_k)_{k>k^\star}$ is bounded in $\mathbb{R}$.
    \item \label{ass:bounded_profiles} Any family of wave profiles $(\boldsymbol{\Phi}_k)_{k>k^\star}$ is locally relatively compact in a H\"{o}lder space.
\end{enumerate}

Note that the systems \eqref{eq:LV_system}, \eqref{eq:PP_system}, \eqref{eq:SKT_system} all satisfy
directly the assumptions \ref{ass:regularity}--\ref{ass:symmetric_competition}. 
As will be explained in the discussion (cf. Section \ref{sec:ordered_monostable_minimal_speeds}),
they also satisfy \ref{ass:ordered_monostable_minimal_speeds}.
It will be explained in the discussion (cf. Section \ref{sec:bounded_profiles}) that
\ref{ass:bounded_profiles} is satisfied for \eqref{eq:LV_system}, \eqref{eq:SKT_system} and \eqref{eq:PP_system} if
$\min\left( \gamma_1,\gamma_2 \right)\geq 0$ (attractive--attractive case). 
Finally, \ref{ass:existence_waves} and \ref{ass:bounded_speeds} are satisfied for
\eqref{eq:LV_system} \cite{Gardner_1982,Kan_on_1995} but are, to the best of our knowledge, 
important open problems for the strongly coupled systems
\eqref{eq:PP_system} and \eqref{eq:SKT_system} that have been solved only in special perturbative 
cases
(cf. Sections \ref{sec:existence_waves} and \ref{sec:bounded_speeds}). Solving these problems away from any such
perturbative regime is way outside the scope of this paper.

\subsection{Statement of the main result}
The main result follows. The limiting profiles and their basic properties -- spatial segregation and free boundary relation -- 
are illustrated on Figure \ref{fig:profiles}.

\begin{thm}\label{thm:main}
    Assume \ref{ass:regularity}--\ref{ass:bounded_profiles}.

    Then there exists $c_\infty\in\mathbb{R}$ and 
    $\boldsymbol{\Phi}_\infty\in \mathscr{C}^2(\mathbb{R}\backslash\left\{ 0 \right\})\cap W^{1,\infty}(\mathbb{R})$ such that,
    up to shifting $(\boldsymbol{\Phi}_k)_{k>k^\star}$,
    \begin{equation*}
	\lim_{k\to+\infty}\left( \left|c_k-c_\infty\right|
	+\left\|\boldsymbol{\Phi}_k-\boldsymbol{\Phi}_\infty\right\|_{L^\infty(\mathbb{R})}
	+\left\|\boldsymbol{\Phi}_k'-\boldsymbol{\Phi}_\infty'\right\|_{L^1(\mathbb{R})} \right) =0.
    \end{equation*}

    The limit satisfies $\phi_\infty \psi_\infty = 0$ and 
    the pair $(\alpha\phi_\infty-\psi_\infty,c_\infty)$ is the unique solution $(z,c)$ of
    \eqref{eq:singular_limit} satisfying $z(0)=0$.

    Consequently, 
    \begin{enumerate}
	\item the profiles $\phi_\infty$ and $\psi_\infty$ are respectively decreasing in $(-\infty,0)$ with
	    limit $1$ at $-\infty$ and increasing in $(0,+\infty)$ with limit $1$ at $+\infty$;
	\item the free boundary relation reads 
	    \begin{equation*}
		-d_{1,1}(0,0)\alpha\lim_{\xi\to 0^-}\phi_\infty'(\xi)=d_{2,2}(0,0)\lim_{\xi\to 0^+}\psi_\infty'(\xi);
	    \end{equation*}
	\item the speed $c_\infty$ satisfies $-c_\star^-<c_\infty<c_\star^+$;
	\item if $h_{1,1}=h_{2,2}=h_0\in\mathbb{R}$, then
	    \begin{equation}\label{eq:sign_speed}
		\operatorname{sign}\left(c_\infty+h_0\right)=\operatorname{sign}\left(\alpha^2 -\frac{\int_{0}^1 zd_{2,2}\left( 0,z \right)g_2\left( 0,z \right)\textup{d}z}{\int_{0}^1 zd_{1,1}\left( z,0 \right)g_1\left( z,0 \right)\textup{d}z}\right).
\end{equation}
    \end{enumerate}
\end{thm}

Below, we specify the theorem for the three main examples. It turns out that 
\eqref{eq:LV_system} and \eqref{eq:PP_system} have the same singular limit whereas
\eqref{eq:SKT_system} has a different limit.

\begin{figure}
    \centering
    \resizebox{\linewidth}{!}{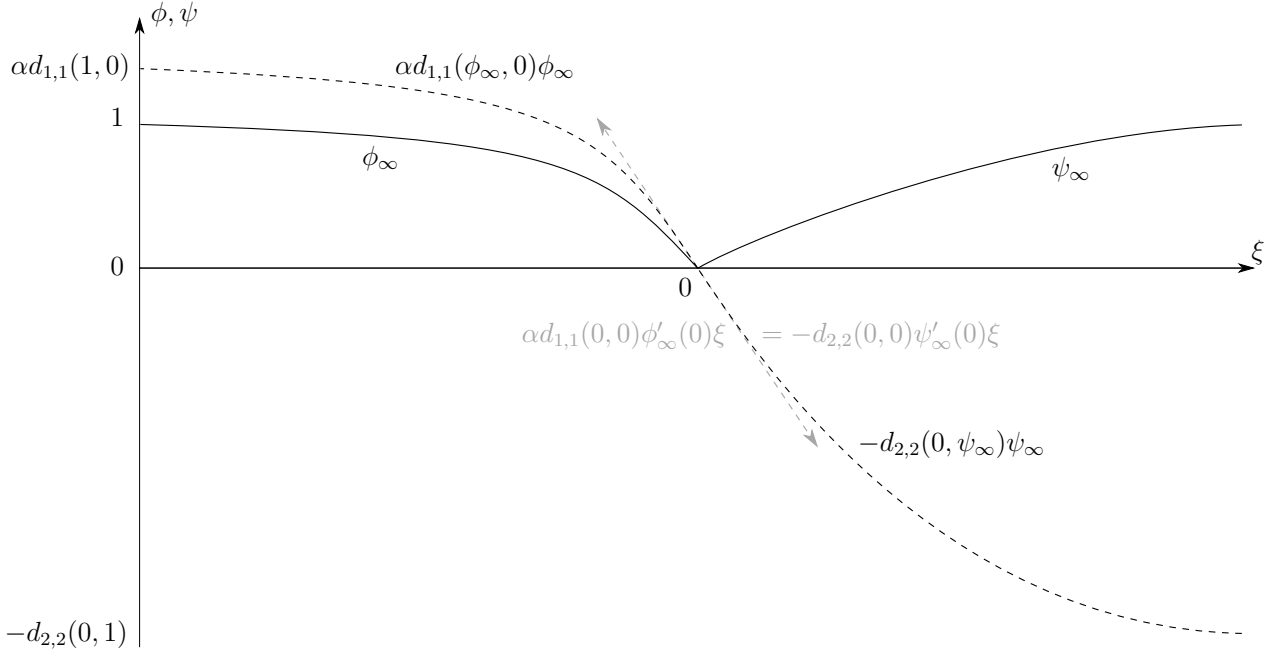}
    \caption{The segregated limiting profiles in the moving frame $\xi=x-c_\infty t$}
    \label{fig:profiles}
\end{figure}

\begin{cor}[``Disunity is strength'' \cite{Girardin_Nadin_2015}]\label{cor:LV_DIS}
    Assume \eqref{eq:RD_system} has the form \eqref{eq:LV_system}.

    Then there exists $c_\infty\in\mathbb{R}$ and 
    $\boldsymbol{\Phi}_\infty\in \mathscr{C}^2(\mathbb{R}\backslash\left\{ 0 \right\})\cap W^{1,\infty}(\mathbb{R})$ such that,
    up to shifting $(\boldsymbol{\Phi}_k)_{k>k^\star}$,
    \begin{equation*}
	\lim_{k\to+\infty}\left( \left|c_k-c_\infty\right|
	+\left\|\boldsymbol{\Phi}_k-\boldsymbol{\Phi}_\infty\right\|_{L^\infty(\mathbb{R})}
	+\left\|\boldsymbol{\Phi}_k'-\boldsymbol{\Phi}_\infty'\right\|_{L^1(\mathbb{R})} \right) =0.
    \end{equation*}

    The limit satisfies $\phi_\infty \psi_\infty = 0$ and 
    the pair $(\alpha\phi_\infty-\psi_\infty,c_\infty)$ is the unique solution $(z,c)$ of
    \begin{equation*}
        \begin{cases}
    	\int-((\mathbf{1}_{z>0}+d\mathbf{1}_{z<0})z')'\chi -c\int z'\chi =\int(z^+(1-\frac{z}{\alpha})-rz^-(1+z))\chi \\
	\quad\text{for all }\chi\in L^1(\mathbb{R}),\\
    	\lim_{-\infty}z=\alpha, \\
    	\lim_{+\infty}z=-1,\\
	z'\leq 0,\\
	z(0)=0.
        \end{cases}
    \end{equation*}

    Consequently, 
    \begin{enumerate}
	\item the profiles $\phi_\infty$ and $\psi_\infty$ are respectively decreasing in $(-\infty,0)$ with
	    limit $1$ at $-\infty$ and increasing in $(0,+\infty)$ with limit $1$ at $+\infty$;
	\item the free boundary relation reads 
	    \begin{equation*}
		-\alpha\lim_{\xi\to 0^-}\phi_\infty'(\xi)=d\lim_{\xi\to 0^+}\psi_\infty'(\xi);
	    \end{equation*}
	\item the speed $c_\infty$ satisfies $-2\sqrt{rd}<c_\infty<2$;
	\item $\operatorname{sign}\left( c_\infty \right)=\operatorname{sign}\left( \alpha^2-rd \right)$.
    \end{enumerate}
\end{cor}

\begin{cor}
    Assume \eqref{eq:RD_system} has the form \eqref{eq:PP_system} and assume furthermore
    \ref{ass:existence_waves}--\ref{ass:bounded_speeds} and, if $\min\left( \gamma_1,\gamma_2 \right)<0$, assume \ref{ass:bounded_profiles}.

    Then all the conclusions of Corollary \ref{cor:LV_DIS} remain true.
\end{cor}

\begin{rem}
The above statement is consistent with the fact that \ref{ass:bounded_profiles} is automatic in the attractive--attractive case 
($\min \gamma_i \geq 0$), as explained in Section \ref{sec:bounded_profiles}.
\end{rem}

\begin{cor}
    Assume \eqref{eq:RD_system} has the form \eqref{eq:SKT_system} and assume furthermore
    \ref{ass:existence_waves}--\ref{ass:bounded_speeds}.

    Then there exists $c_\infty\in\mathbb{R}$ and 
    $\boldsymbol{\Phi}_\infty\in \mathscr{C}^2(\mathbb{R}\backslash\left\{ 0 \right\})\cap W^{1,\infty}(\mathbb{R})$ such that,
    up to shifting $(\boldsymbol{\Phi}_k)_{k>k^\star}$,
    \begin{equation*}
	\lim_{k\to+\infty}\left( \left|c_k-c_\infty\right|
	+\left\|\boldsymbol{\Phi}_k-\boldsymbol{\Phi}_\infty\right\|_{L^\infty(\mathbb{R})}
	+\left\|\boldsymbol{\Phi}_k'-\boldsymbol{\Phi}_\infty'\right\|_{L^1(\mathbb{R})} \right) =0.
    \end{equation*}

    The limit satisfies $\phi_\infty \psi_\infty = 0$ and 
    the pair $(\alpha\phi_\infty-\psi_\infty,c_\infty)$ is the unique solution $(z,c)$ of
    \begin{equation*}
        \begin{cases}
    	\int-(((d_{1}+a_{1,1}z)\mathbf{1}_{z>0}+(d_2+a_{2,2}z)\mathbf{1}_{z<0})z')'\chi-c\int z'\chi=\int(z^+(1-\frac{z}{\alpha})-rz^-(1+z))\chi \\
	\quad\text{for all }\chi\in L^1(\mathbb{R}),\\
    	\lim_{-\infty}z=\alpha, \\
    	\lim_{+\infty}z=-1,\\
	z'\leq 0,\\
	z(0)=0.
        \end{cases}
    \end{equation*}

    Consequently, 
    \begin{enumerate}
	\item the profiles $\phi_\infty$ and $\psi_\infty$ are respectively decreasing in $(-\infty,0)$ with
	    limit $1$ at $-\infty$ and increasing in $(0,+\infty)$ with limit $1$ at $+\infty$;
	\item the free boundary relation reads 
	    \begin{equation*}
		-d_1\alpha\lim_{\xi\to 0^-}\phi_\infty'(\xi)=d_2\lim_{\xi\to 0^+}\psi_\infty'(\xi);
	    \end{equation*}
	\item the speed $c_\infty$ satisfies $-c_\star^-<c_\infty<c_\star^+$, where $c_\star^+$
	    is the minimal speed of monotonic traveling wave solutions of
	    \begin{equation*}
		-((d_{1}+a_{1,1}z)z')'-cz'=z\left(1-\frac{z}{\alpha}\right),\quad z(-\infty)=\alpha,\quad 
		z(+\infty)=0,
	    \end{equation*}
	    and $c_\star^-$ is the minimal speed of monotonic traveling wave solutions of
	    \begin{equation*}
		-((d_{2}+a_{2,2}z)z')'-cz'=rz(1-z),\quad z(-\infty)=1,\quad z(+\infty)=0;
	    \end{equation*}
	\item $\operatorname{sign}\left( c_\infty \right)=\operatorname{sign}\left( \alpha^2(d_1+a_{1,1})-r(d_2+a_{2,2}) \right)$.
    \end{enumerate}
\end{cor}

\begin{rem}
In the preceding corollary, if $d_1\geq a_{1,1}$, then $z\mapsto(d_1+a_{1,1}z)z(1-z)$ is concave and
therefore the speed $c_\star^+$ is linearly determined \cite{Malaguti_Marcelli_2002}: 
$c_\star^+ = 2\sqrt{d_1}$. 
Similarly, if $d_2\geq a_{2,2}$, then $c_\star^-=2\sqrt{rd_2}$.
\end{rem}

\subsection{Organization of the paper}
In Section 2 we prove the main result.
In Section 3 we discuss at length the literature, the assumptions and the results. 

\section{Proof of Theorem \ref{thm:main}}

For clarity, we divide the proof in several parts.

\subsection{Existence of limit points}
\begin{proof}
Up to extraction, the bounded family $(c_k)_{k>k^\star}$ converges to some limit $c_\infty$ (whose uniqueness is unclear at this point).

Now, we fix a family of shifts for the profiles:
\begin{enumerate}
    \item if $c_\infty\leq \frac12(c_\star^+-c_\star^-)$, the family $(\boldsymbol{\Phi})_{k>k^\star}$ is shifted so that $\phi_k(0)=\frac12$;
    \item if $c_\infty> \frac12(c_\star^+-c_\star^-)$, the family $(\boldsymbol{\Phi})_{k>k^\star}$ is shifted so that $\psi_k(0)=\frac12$.
\end{enumerate}

Next, since the family of profiles is locally relatively compact in some H\"{o}lder space, up to another (diagonal) extraction, 
$(\boldsymbol{\Phi}_k)_{k>k^\star}$ converges locally uniformly to
some limit $\boldsymbol{\Phi}_\infty$ (whose uniqueness is also unclear at this point). 

Since the families of profile derivatives are uniformly bounded in $L^1(\mathbb{R})$,
by the Banach--Alaoglu theorem, up to another
extraction, the families $(\phi_k'\textup{d}\xi)_{k>k^\star}$ and $(\psi_k'\textup{d}\xi)_{k>k^\star}$ converge in the 
weak-$\star$ topology of the set of bounded Radon measures to some limits, $\textup{d}\phi_\infty$ and $\textup{d}\psi_\infty$ 
respectively, which are supported in the support of $\phi_\infty$ and $\psi_\infty$ respectively.

We deduce from such a convergence that $\phi_\infty$ and $\psi_\infty$ are continuous, valued in $[0,1]$, 
nonincreasing and nondecreasing respectively and with (distributional) derivatives that are bounded Radon measures of total mass smaller 
than or equal to $1$. Subsequently, $\phi_\infty$ and $\psi_\infty$ have limits at $\pm\infty$.
\end{proof}

\subsection{Segregation of the limiting supports}
\begin{proof}
Testing for all $k>k^\star$ one of the two equations of \eqref{eq:TW_system}, say the first one, against any smooth and compactly 
supported test function $\chi\in\mathscr{D}(\mathbb{R})$ and using the dominated convergence theorem as well as estimates of the form
\begin{align*}
    \left| \int_{\mathbb{R}} d_{1,2}(\phi_k,\psi_k)\psi_k' \chi'  \right| & \leq \left\|d_{1,2}\right\|_{L^\infty([0,1]^2)}\left\|\chi'\right\|_{L^\infty(\mathbb{R})}\int_K \left|\psi_k'\right| \\
    & =\left\|d_{1,2}\right\|_{L^\infty([0,1]^2)}\left\|\chi'\right\|_{L^\infty(\mathbb{R})}\int_K \psi_k' \\
    & \leq\left\|d_{1,2}\right\|_{L^\infty([0,1]^2)}\left\|\chi'\right\|_{L^\infty(\mathbb{R})},
\end{align*}
where $K\subset\mathbb{R}$ is the support of $\chi$, we find that 
$\omega(\boldsymbol{\Phi}_k)$ converges to $0$ in the space of distributions. 
By continuity of $\omega$, $\omega(\boldsymbol{\Phi}_\infty)=0$, whence at any $\xi\in\mathbb{R}$, $\phi_\infty(\xi)=0$ or 
$\psi_\infty(\xi)=0$. Consequently, the following equalities are true pointwise in $\mathbb{R}$:
\begin{enumerate}
    \item $\alpha\phi_\infty=\left( \alpha\phi_\infty-\psi_\infty \right)^+$;
    \item $\psi_\infty=\left( \alpha\phi_\infty-\psi_\infty \right)^-$.
\end{enumerate}

Recall that $\textup{d}\psi_\infty$ is supported in the support of $\psi_\infty$, namely the 
closure of $\psi_\infty^{-1}\left( (0,1] \right)$. Note
that $\psi_\infty^{-1}\left( (0,1] \right)\subset\phi_\infty^{-1}(\left\{ 0 \right\})$. Since 
$\boldsymbol{\Phi}_\infty$ and $d_{1,2}$ are continuous with $d_{1,2}(0,\bullet)=0$,
\begin{equation*}
    \int_{\mathbb{R}}d_{1,2}(\boldsymbol{\Phi}_\infty)\textup{d}\psi_\infty =
    \int_{\psi_\infty^{-1}( (0,1])\cup\partial(\psi_\infty^{-1}( (0,1]))}d_{1,2}(0,\psi_\infty)\textup{d}\psi_\infty = 0.
\end{equation*}
Note that this calculation does not require the absolute continuity of $\textup{d}\phi_\infty$ with respect to
the Lebesgue measure, which is unclear \textit{a priori}.
Repeating this argument, we discover that the equalities 
\begin{equation}\label{eq:null_terms}
    \int_{\mathbb{R}}d_{1,2}(\boldsymbol{\Phi}_\infty)\textup{d}\psi_\infty=\int_{\mathbb{R}}h_{1,2}(\boldsymbol{\Phi}_\infty)\textup{d}\psi_\infty=\int_{\mathbb{R}}d_{2,1}(\boldsymbol{\Phi}_\infty)\textup{d}\phi_\infty=\int_{\mathbb{R}}h_{2,1}(\boldsymbol{\Phi}_\infty)\textup{d}\phi_\infty=0
\end{equation}
are all true.
\end{proof}

\subsection{The limiting equation}
\begin{proof}
For all $k>k^\star$, the scalar equation
\begin{equation*}
    \left[-(\mathbf{D}(\boldsymbol{\Phi}_k)\boldsymbol{\Phi}_k')'-(\mathbf{H}(\boldsymbol{\Phi}_k)+c_k\mathbf{I})\boldsymbol{\Phi}_k' -
    \boldsymbol{\Phi}_k\circ\mathbf{g}(\boldsymbol{\Phi}_k) + k\omega(\boldsymbol{\Phi}_k)
    \begin{pmatrix}
	1 \\
	\alpha
    \end{pmatrix}
    \right]\cdot
    \begin{pmatrix}
	\alpha \\
	-1
    \end{pmatrix}
    =0
\end{equation*}
is equivalent to
\begin{equation}\label{eq:linear_comb}
    \left[-(\mathbf{D}(\boldsymbol{\Phi}_k)\boldsymbol{\Phi}_k')'-(\mathbf{H}(\boldsymbol{\Phi}_k)+c_k\mathbf{I})\boldsymbol{\Phi}_k' -
    \boldsymbol{\Phi}_k\circ\mathbf{g}(\boldsymbol{\Phi}_k)
    \right]\cdot
    \begin{pmatrix}
	\alpha \\
	-1
    \end{pmatrix}
    =0.
\end{equation}
The equation \eqref{eq:linear_comb} does not depend explicitly on $k$. Our aim is to pass to the limit in this equation.
Therefore we test this equation against a smooth and compactly supported test function $\chi\in\mathscr{D}(\mathbb{R})$ 
and we subtract from the result the equation \eqref{eq:singular_limit} tested against the same $\chi$. Working in the space
of distributions, we can write $\phi_\infty'$ and $\psi_\infty'$ without abuse of notation.

The zeroth-order terms converge to $0$ by the dominated convergence convergence.

The first-order terms converge to $0$ by estimates of the form
\begin{align*}
    \left| \int_{\mathbb{R}} (h_{1,2}(\boldsymbol{\Phi}_k)\psi_k'-h_{1,2}(\boldsymbol{\Phi}_\infty)\psi_\infty') \chi \right| & 
    \leq \left\|h_{1,2}(\boldsymbol{\Phi}_k)-h_{1,2}(\boldsymbol{\Phi}_\infty)\right\|_{L^\infty(K)}\left\|\chi\right\|_{L^\infty(\mathbb{R})}\int_K \left|\psi_k'\right| \\
    & \quad + \left|\int_{\mathbb{R}} h_{1,2}(\boldsymbol{\Phi}_\infty)\chi\left(\psi_k'-\psi_\infty'\right)\right|
\end{align*}
(where the compact set $K\subset\mathbb{R}$ is the support of $\chi$).
The first term on the right-hand side converges to $0$ by $\int_K \left|\psi_k'\right|\leq 1$ and 
locally uniform convergence of the profiles whereas
the second term converges to $0$ by weak-$\star$ convergence of the derivatives (the test function 
$\left|h_{1,2}(\boldsymbol{\Phi}_\infty)\chi\right|$ is indeed continuous and zero at infinity).

The second-order terms converge to $0$ by similar estimates with $\chi$ replaced by $\chi'$.

Eliminating all the terms that are actually zero \eqref{eq:null_terms}, we find that 
$(z,c)=(\alpha\phi_\infty-\psi_\infty,c_\infty)$ is a solution of
\begin{equation}\label{eq:weak_singular_limit}
    \int d(z)z'\chi'-\int(c+h(z))z'\chi = \int g(z)\chi\quad\text{for all }\chi\in\mathscr{D}(\mathbb{R}),
\end{equation}
where the functions $d$, $h$ and $g$ are defined above in \eqref{def:d}, \eqref{def:h}, \eqref{def:g}.

Note that, when identifying the terms, it can be useful to cut all integrals into three parts: the integral over the set
$z^{-1}\left( \left( 0,\alpha \right] \right)$, the one over $z^{-1}\left( \left\{ 0 \right\} \right)$
and the one over $z^{-1}\left( \left[ -1,0 \right) \right)$. By monotonicity of $z$, these three sets are
intervals, ordered from left to right. Also, some of these sets might be empty (at this point).
It turns out that all integrals over the closed interval $z^{-1}\left( \left\{ 0 \right\} \right)$ are zero. Indeed,
\begin{enumerate}
    \item if this interval is empty or a singleton, then the set is negligible;
    \item if this interval is neither empty nor a singleton, then $z=z'=0$ there.
\end{enumerate}
\end{proof}

\begin{rem}
    In \cite{Girardin_Nadin_2015}, the limiting equation (of the form 
    $-(\alpha\phi_\infty-d\psi_\infty)''-c(\alpha\phi_\infty-\psi_\infty)'=g(\alpha\phi_\infty-\psi_\infty)$ in 
    $\mathscr{D}'(\mathbb{R})$, with a constant diffusion rate $d>0$) was used to deduce directly
    the continuity of $(\alpha\phi_\infty-d\psi_\infty)'$, just by taking the antiderivative of the equation. 
    Let us emphasize that here the situation is not so easy: indeed, at this point of the proof,
    we do not know that the distributional derivative $z'$ is sufficiently regular to write the chain rule $h(z)z'=(H(z))'$ 
    (where $H$ is a Lipschitz-continuous antiderivative of $h$). 
    Therefore more care is needed to obtain the regularity at the free boundary and subsequently the free boundary relation. 
    In order to complete the proof, we will use ideas from \cite{Girardin_Nadin_2016}, where the free boundary problem
    was also more delicate, due to the spatial periodicity of the setting there.
\end{rem}

\subsection{Non-triviality, piecewise-smoothness and limits at infinity of the limit points}
\begin{proof}
By pointwise convergence,
\begin{enumerate}
    \item if $c_\infty\leq \frac12(c_\star^+-c_\star^-)$, $\phi_\infty(0)=\frac12$;
    \item if $c_\infty> \frac12(c_\star^+-c_\star^-)$, $\psi_\infty(0)=\frac12$.
\end{enumerate}
Note that, by \ref{ass:ordered_monostable_minimal_speeds}, 
\begin{equation}\label{eq:ordered_speeds_middle}
    -c_\star^-<\frac12(c_\star^+-c_\star^-)<c_\star^+.
\end{equation}

We define $\xi^+=\sup \phi_\infty^{-1}( (0,1] )$ and $\xi^-=\inf \psi_\infty^{-1}( (0,1] )$. These two quantities,
\textit{a priori} defined in $\overline{\mathbb{R}}$, satisfy $\xi^+\leq\xi^-$ by virtue of $\phi_\infty \psi_\infty =0$.
Moreover, $\xi^+\geq 0$ if $c_\infty\leq \frac12(c_\star^+-c_\star^-)$ and $\xi^-\leq 0$ if 
$c_\infty> \frac12(c_\star^+-c_\star^-)$.

We are going to prove now the following claim:
\begin{equation}
    \xi^\pm\in\mathbb{R},\ \phi_\infty\in \mathscr{C}^2\left( (-\infty,\xi^+) \right),\ 
    \psi_\infty\in \mathscr{C}^2\left( (\xi^-,+\infty) \right),\ \lim_{-\infty}\phi_\infty=\lim_{+\infty}\psi_\infty=1.
    \label{eq:claim_proof}
\end{equation}

Assume for instance $c_\infty\leq \frac12(c_\star^+-c_\star^-)$, so that $\xi^+\geq 0$. 

From the equation \eqref{eq:weak_singular_limit} tested against test functions supported in $(-\infty,\xi^+)$, where $\psi_\infty$ is 
identically zero, we find that $\phi_\infty$ is a weak solution of
\begin{equation}\label{eq:limiting_equation_on_the_left}
    -(d_{1,1}(\phi_\infty,0)\phi_\infty')'-(c_\infty+h_{1,1}(\phi_\infty,0))\phi_\infty'=\phi_\infty g_1(\phi_\infty,0)\quad\text{in }
    (-\infty,\xi^+).
\end{equation}
By standard elliptic regularity \cite{Gilbarg_Trudin}, $\phi_\infty$ is then $\mathscr{C}^2$ in $(-\infty,\xi^+)$. 

Since $(-\infty,\xi^+)$ is not empty and $\phi_\infty$ is non-increasing and larger than or equal
to $1/2$ in $(-\infty,0)$, its limit $C$ at $-\infty$ is well-defined and larger than or equal to
$1/2$. Let $(\xi_n)_{n\in\mathbb{N}}$ be a sequence such that $\xi_n\to-\infty$ as $n\to+\infty$
and define, for all $n\in\mathbb{N}$, $\phi^n_\infty:\xi\mapsto\phi_\infty(\xi+\xi_n)$. 
By standard elliptic estimates, up to extraction of a subsequence, the sequence 
$(\phi_\infty^n)_{n\in\mathbb{N}}$ converges in $\mathscr{C}^2_{\textup{loc}}$. By uniqueness
of the limit, the limit in $\mathscr{C}^2_{\textup{loc}}$ is also the constant $C\geq 1/2$.
By \eqref{eq:limiting_equation_on_the_left}, invariance by translation and convergence in 
$\mathscr{C}^2_{\textup{loc}}$, we directly deduce that $C\geq 1/2$ satisfies $Cg_1(C,0)=0$, whence
$C=1$ by \ref{ass:monostability}.

In order to prove the finiteness of $\xi^+$, assume by contradiction $\xi^+=+\infty$. 
Then $\psi_\infty=0$ identically and $\phi_\infty(\xi)>0$ at any $\xi\in\mathbb{R}$.
Hence $\phi_\infty$ is a positive weak solution, and then a positive classical solution by elliptic regularity, of
\begin{equation*}
    -(d_{1,1}(\phi_\infty,0)\phi_\infty')'-(c_\infty+h_{1,1}(\phi_\infty,0))\phi_\infty'=\phi_\infty g_1(\phi_\infty,0)\quad\text{in }\mathbb{R}.
\end{equation*}
Repeating the argument, the limit of $\phi_\infty$ at $+\infty$ is then $0$. 
Hence $(t,x)\mapsto \phi_\infty(x-c_\infty t)$ is a traveling wave solution with speed $c_\infty$, and by
\eqref{eq:ordered_speeds_middle}, $c_\infty<c_\star^+$, which contradicts 
the analysis of monostable equations by Malaguti and Marcelli \cite{Malaguti_Marcelli_2002}. 
Therefore $\xi^+<+\infty$, or in other words $\xi^+\in\mathbb{R}$.

Next, integrating the equation satisfied by 
$\phi_\infty$ in $(\xi^+-1,\xi^+)$, we infer that the left-sided limit of $\phi_\infty'$ at
$\xi^+$ is finite and, by the Hopf lemma, negative. Since $\xi^-\in\left[ \xi^+,+\infty \right]$, we infer similarly that
one, and only one, of the following two claims is true:
\begin{enumerate}
    \item $\xi^-=+\infty$,
    \item $\psi_\infty$ is $\mathscr{C}^2$ in $(\xi^-,+\infty)$ with $1$ as limit at $+\infty$ and a finite and 
	positive right-sided limit of $\psi_\infty'$ at $\xi^-$. 
\end{enumerate}

Integrating the equation \eqref{eq:weak_singular_limit} over $\mathbb{R}$ with
an arbitrary test function $\chi\in\mathscr{D}(\mathbb{R})$ and using the fact that $z=0$
in $(\xi^+,\xi^-)$, we obtain:
\begin{align*}
    0 & = \int_{-\infty}^{+\infty}\left[g(z)\chi+(c_\infty+h(z))z'\chi-d(z)z'\chi'\right] \\
    & = \int_{-\infty}^{\xi^+}\left[g(z)\chi+(c_\infty+h(z))z'\chi-d(z)z'\chi'\right] + \int_{\xi^-}^{+\infty}\left[g(z)\chi+(c_\infty+h(z))z'\chi-d(z)z'\chi'\right]
\end{align*}
Integrating by parts the terms $\int_{-\infty}^{\xi^+}-d(z)z'\chi'$ and
$\int_{\xi^-}^{+\infty}-d(z)z'\chi'$, we deduce:
\begin{equation}\label{eq:free_boundary_relation_test_function}
    -d_{1,1}(0,0)\chi(\xi^+)\alpha\lim_{\xi-\xi^+\to0^-}\phi_\infty'(\xi)=
    \begin{cases}
	\displaystyle d_{2,2}(0,0)\chi(\xi^-)\lim_{\xi-\xi^-\to0^+}\psi_\infty'(\xi) & \text{if }\xi^-<+\infty,\\
	0 & \text{if }\xi^-=+\infty.	
    \end{cases}
\end{equation}
By positivity of the left-hand side, necessarily $\xi^-<+\infty$, or in other words $\xi^-\in\mathbb{R}$. 
We have thus proved the claim \eqref{eq:claim_proof} in the case $c_\infty\leq \frac12(c_\star^+-c_\star^-)$.

Repeating the argument with the roles of $\phi_\infty$ and $\psi_\infty$ reversed, we find
that the claim \eqref{eq:claim_proof} remains true even if $c_\infty< \frac12(c_\star^+-c_\star^-)$, that is
for any possible value of $c_\infty$.
\end{proof}

\subsection{Triviality of the free boundary, the free boundary relation and regularity of the first derivatives}
\begin{proof}
It follows from \eqref{eq:free_boundary_relation_test_function} that
\begin{equation*}
    d_{1,1}(0,0)\chi(\xi^+)\alpha\left(-\lim_{\xi-\xi^+\to0^-}\phi_\infty'(\xi)\right)
    =d_{2,2}(0,0)\chi(\xi^-)\lim_{\xi-\xi^-\to0^+}\psi_\infty'(\xi)
\end{equation*}
with $\lim_{\xi-\xi^+\to0^-}\phi_\infty'(\xi)<0<\lim_{\xi-\xi^-\to0^+}\psi_\infty'(\xi)$.

We first check that $\xi^+=\xi^-$. Indeed, if $\xi^+<\xi^-$, we can choose the test function
$\chi$ so that its support contains $\xi^+$ but not $\xi^-$, and then the above equality
becomes an equality between a positive term and a zero term, which is contradictory.
Therefore $\xi^+=\xi^-$. 

Next, choosing $\chi$ so that its support contains $\xi^+=\xi^-$ and dividing by 
$\chi(\xi^+)\neq 0$, we find
\begin{equation*}
    d_{1,1}(0,0)\alpha\left(-\lim_{\xi-\xi^+\to0^-}\phi_\infty'(\xi)\right)
    =d_{2,2}(0,0)\lim_{\xi-\xi^-\to0^+}\psi_\infty'(\xi).
\end{equation*}
In other words, $d(z)z'$ is a well-defined continuous function in $\mathbb{R}$.

In order to simplify the notations, we work from now on with the shifted profiles so 
that $\xi^+=\xi^-=0$.
With this convention, the free boundary relation reads 
\begin{equation*}
    -d_{1,1}(0,0)\alpha\lim_{\xi\to 0^-}\phi_\infty'(\xi)=d_{2,2}(0,0)\lim_{\xi\to 0^+}\psi_\infty'(\xi).
\end{equation*}

Since the limits of $d(z)z'$ at $\pm\infty$ are zero, we also obtain $d(z)z'\in L^\infty(\mathbb{R})$.
We deduce the piecewise-continuity and global boundedness of $z'$. Hence the distribution $z$ is a well-defined function in 
$W^{1,\infty}(\mathbb{R})$, or in other words a well-defined Lipschitz-continuous, piecewise-$\mathscr{C}^1$, function in $\mathbb{R}$. 

Since $z$ has no jump discontinuity at $0$, a standard result of distribution theory yields the identification 
$z'\textup{d}\xi=\alpha\textup{d}\phi_\infty-\textup{d}\psi_\infty$. 
This proves that the measures $\textup{d}\phi_\infty$ and $\textup{d}\psi_\infty$ are 
absolutely continuous with respect to the Lebesgue measure, or in other words that the derivatives $\phi_\infty'$ and $\psi_\infty'$ are 
well-defined functions in $L^1(\mathbb{R})$. 
\end{proof}

\subsection{Regularity of the second derivatives}
\begin{proof}
    From the equality
\begin{equation*}
    (d(z)z')'=-(c+h(z))z'-g(z)\quad\text{in }\mathscr{D}'(\mathbb{R}),
\end{equation*}
and the fact that the right-hand side is a well-defined function in $L^\infty(\mathbb{R})$,
we infer that $(d(z)z')'\in L^\infty(\mathbb{R})$ (in other words, $d(z)z'$ is 
Lipschitz-continuous).

We are now in a position to rewrite \eqref{eq:weak_singular_limit} with less regular test 
functions: by density,
\begin{equation*}
    \int -(d(z)z')'\chi-\int(c+h(z))z'\chi = \int g(z)\chi\quad\text{for all }\chi\in L^1(\mathbb{R}).
\end{equation*}
\end{proof}

\subsection{Uniqueness of the limit point}
\begin{proof}
The uniqueness (up to shifts) of the limit point is a direct consequence of the uniqueness (up to shifts) of the bistable
traveling wave (Theorem \ref{thm:app}).

By a classical compactness argument, the uniqueness of the limit point implies the convergence of the full initial family
$(\boldsymbol{\Phi}_k,c_k)_{k>k^\star}$ (as defined before any extraction of subsequence).
\end{proof}

\subsection{Estimates on $c_\infty$}
\begin{proof}
Similarly, the estimates $-c_\star^-<c_\infty<c_\star^+$ follow directly from Theorem \ref{thm:app}.

The sign of the effective wave speed if $h_{1,1}=h_{2,2}=h_0\in\mathbb{R}$ (which implies $h=h_0$) is obtained by 
testing \eqref{eq:singular_limit} against $d(z)z'$. After an integration by parts and a change of variable, we find
\begin{equation*}
    (c_\infty+h_0)\int_{\mathbb{R}}d(z)(z')^2=
    \alpha^2 \int_{0}^1 \hat{z}d_{1,1}\left( \hat{z},0 \right)g_1\left( \hat{z},0 \right)\textup{d}\hat{z}
    -\int_{0}^1 \hat{z}d_{2,2}\left( 0,\hat{z} \right)g_2\left( 0,\hat{z} \right)\textup{d}\hat{z},
\end{equation*}
so that the effective wave speed $c_\infty+h_0$ has indeed the sign of the right-hand side.
Note that since $z'\in L^1(\mathbb{R})\cap L^\infty(\mathbb{R})$, clearly $z'\in L^2(\mathbb{R})$.
\end{proof}

\subsection{Improvement of the convergence}
\begin{proof}
    In view of the limits at $\pm\infty$ of $\phi_\infty$ and $\psi_\infty$, the equalities $\int \phi_k'=-1$ and
$\int \psi_k'=1$ are true for all $k\in(k^\star,+\infty]$. From this remark, it follows that:
\begin{enumerate}
    \item by a result sometimes known as the second Dini theorem \footnote{The second Dini theorem
	states that a sequence of nondecreasing functions that converges pointwise in a closed
	interval of $\overline{\mathbb{R}}$ to a continuous function converges uniformly in this 
	interval; when $\pm\infty$ belong to the interval, the assumption is naturally understood 
	as the convergence of the limits at $\pm\infty$ to the limits at $\pm\infty$ of the
        continuous limit.},
	the convergence of the families of monotonic functions $(\phi_k)_{k>k^\star}$ and
	$(\psi_k)_{k>k^\star}$ 	is actually uniform in $\mathbb{R}$;
    \item by the Prokhorov theorem, the convergence of the tight families of probability measures
	$(-\phi_k'\textup{d}\xi)_{k>k^\star}$ and $(\psi_k'\textup{d}\xi)_{k>k^\star}$ 
	actually occurs in the weak-$\star$ topology of the set of probability measures in 
	$\mathbb{R}$ (\textit{i.e.}, the measures converge when tested against any fixed
	continuous bounded function from $\mathbb{R}$ to $\mathbb{R}$). 
	Since the limits $-\phi_\infty'\textup{d}\xi$ and 
	$\psi_\infty'\textup{d}\xi$ are absolutely continuous with respect to the Lebesgue 
	measure, this implies the strong convergence of $(\boldsymbol{\Phi}_k')_{k>k^\star}$ in 
	$L^1(\mathbb{R})$ (testing as usual the weak-$\star$ convergence against the test 
	function $\mathbf{1}_{\mathbb{R}}$).
\end{enumerate}
\end{proof}

\section{Discussion}

Our focus is on the effect of $\mathbf{D}$ and $\mathbf{H}$. For $\mathbf{g}$ and $\omega$, we mostly have in mind the special 
Lotka--Volterra case:
\begin{equation}\label{eq:LV_reaction}
    \mathbf{g}(\boldsymbol{\Phi})=
    \begin{pmatrix}
	1-\phi \\
	r(1-\psi)
    \end{pmatrix}
    \quad\text{and}\quad\omega(\boldsymbol{\Phi})=\phi\psi.
\end{equation}

\subsection{On the literature}

To the best of our knowledge, this paper is the very first to consider the spatial segregation limit of traveling waves
for general fully nonlinear competitive systems and to show the connection with general fully nonlinear scalar equations.

\subsubsection{Fully nonlinear strongly coupled competitive systems}

The semilinear competition--diffusion system \eqref{eq:TW_LV_system} was studied by the first author and G. Nadin 
in \cite{Girardin_Nadin_2015} 
(refer also to \cite{Girardin_Nadin_2016} for an extension to space-periodic media). The spatial segregation limit
was used to characterize the sign of the wave speed: $c_\infty$ has in such a case the sign of $\alpha^2-rd$.
This was interpreted as a ``Unity is not strength''-type result (and, actually, as a ``Disunity is strength''-type result):
if $\alpha=r=1$, so that the two species only differ in diffusion rate, then the invader is the fast diffuser (that is,
with pure Brownian motion, the winning strategy is to disperse rapidly instead of remaining concentrated).
The problem of the sign of the wave speed for finite values of $k$ is the object of a recent survey paper by the
first author \cite{Girardin_2019}.

The paper \cite{Girardin_Nadin_2015} used extensively the ideas of the wide literature on the spatial segregation limit
of parabolic or elliptic competition--diffusion systems in bounded domains
(\textit{e.g.} \cite{Dancer_Du_1994,Conti_Terracin,Soave_Zilio_2015,Dancer_Hilhors,Crooks_Dancer__1,Dancer_2010,Crooks_Dancer_} and 
references therein).

These ideas were also used in a collection of works by Liu and various collaborators
\cite{Zhang_Zhou_Liu_2017,Zhou_Zhang_Liu_Lin_2012,Zhou_Zhang_Liu_2012} 
studying the spatial segregation limit of the strongly coupled elliptic system \eqref{eq:SKT_system}
with self- and cross-diffusion of Shigesada--Kawasaki--Teramoto type \cite{SKT_79}, where
\eqref{eq:LV_reaction}, $\mathbf{H}=\mathbf{0}$,
\begin{equation*}
    \mathbf{D}=\textup{Jac}
\begin{pmatrix}
    u(d_1+a_{1,1}u+a_{1,2}v) \\
    v(d_2+a_{2,1}u+a_{2,2}v)
\end{pmatrix}
=
\begin{pmatrix}
    d_1+2a_{1,1}u+a_{1,2}v & a_{1,2}u \\
    a_{2,1}v & d_2+a_{2,1}u+2a_{2,2}v
\end{pmatrix},
\end{equation*}
all parameters above being constant and positive.
The main idea there was to change variables in order to recover an elliptic system with linear diffusion.
We point out that although our system can be seen as an elliptic one, the unboundedness of the domain and the generality of the
functions $\mathbf{D}$ and $\mathbf{H}$ make the problem quite different.
Although the Cauchy problem for the cross- and self-diffusion system in an unbounded domain has been studied more than
ten years ago by Dreher \cite{Dreher_2007}, we believe that the traveling wave problem is still largely open (\textit{e.g.}, \cite{Wu_Zhao_2010}).

Bistable traveling waves for the competition--diffusion--cross-taxis system \eqref{eq:PP_system}, where
\begin{equation*}
    \eqref{eq:LV_reaction},\quad
    \mathbf{H}=\mathbf{0},\quad\mathbf{D}=
\begin{pmatrix}
    1 & -\gamma_1 u \\
    -\gamma_2 v & 1 \\
\end{pmatrix}.
\end{equation*}
and with finite, but large, values of $k$ were numerically investigated by Potts and Petrovskii in \cite{Potts_Petrovskii}. 
The cross-taxis is attractive if $\gamma_i>0$ and repulsive if $\gamma_i<0$.
The point there was to show that appropriate choices of $\gamma_i$ can change the sign of the wave speed, 
or equivalently the invader: roughly speaking, aggressivity can compensate unity. 
The authors called this conclusion a ``Fortune favors the bold''-type result.
The well-posedness of this system with $\gamma_1=\gamma_1(u,v)$ and $\gamma_2=0$ was studied independently in a collection of works 
by Wang and various collaborators \cite{Gai_Wang_Yan_2,Wang_Yang_Yu,Wang_Zhang_2017,Zhang_2018} (refer also to \cite{Kubo_Tello_2015} for
a case where $v$ does not diffuse at all). 

We also point out a recent analytical and numerical study by Krause and Van Gorder \cite{Krause_Van_Gorder_2020} of traveling 
wave solutions to
a closely related non-local system with advection toward resource gradients and where each resource density solves an elliptic
PDE involving the two population densities (modeling the consumption of resources). Since the resource density converges to a linear 
combination of both population densities in the vanishing viscosity limit, our model \eqref{eq:RD_system}
can be understood as a singular limit of that model (see \cite[Section 2.1]{Krause_Van_Gorder_2020}).

\subsubsection{Fully nonlinear scalar equations}

Equations of the form 
\begin{equation*}
    -(d(z)z')'-(c+h(z))z' = g(z),
\end{equation*}
are the object of a specific wide literature (\textit{e.g.}
\cite{Malaguti_Marcelli_2002,Marcelli_Papalini_2018,Malaguti_Marcelli_Matucci_2004,Sanchez-Garduno_Maini_1995,Sanchez-Garduno_Maini_1994,Ferracuti_Marc,Maini_Malaguti,Bao_Zhou_2014,Kuzmin_Ruggerini_2011,Malaguti_Marcelli_Matucci_2010}
and references therein).
Generically, the important features of the semilinear equations (existence and uniqueness of waves,
estimates for the minimal wave speed, etc.) are recovered but in a loose sense.

We point out that in this literature, it is typically assumed that $d$, $g$ and $h$ are at least continuous.
This is true in our case when considering the underlying monostable fronts (\textit{i.e.}, connecting $0$ and $\alpha$
or $0$ and $-1$) but, when considering the bistable fronts (connecting $\alpha$ and $-1$), our functions $d$ and $h$ have
a jump discontinuity at $z=0$. It turns out this obstacle can be easily overcome and we will show in 
Appendix \ref{app} how the previously known results can be extended to such cases.
Our method of proof uses extensively the results of Malaguti and her collaborators 
\cite{Malaguti_Marcelli_2002,Malaguti_Marcelli_Matucci_2004} to
characterize the singular limit.

\subsection{On the standing assumptions}

\subsubsection{\ref{ass:regularity}}

This is a very standard regularity assumption. Although it seems to us that the results might remain
true with less regularity, our aim here is not to focus on regularity issues and we prefer to assume that everything is
``smooth enough''. Besides, our proof really needs the continuity of $\mathbf{D}$, $\mathbf{H}$, $\mathbf{g}$ and $\omega$. The continuous differentiability of $\mathbf{D}$ is not explicitly
used in our proof but is required for the continuity result of 
\cite{Malaguti_Marcelli_Matucci_2010} (which is used in the proof of Theorem \ref{thm:app}). 

\subsubsection{\ref{ass:ellipticity}}

This is a very standard ellipticity assumption. Although the main part of our proof does not rely upon the sign of 
$d_{1,1}$ and $d_{2,2}$, we need a strong characterization of the limit points (uniqueness in particular) that is available
indeed in the elliptic framework but might be unavailable in a degenerate elliptic or aggregative framework.
Furthermore, the compactness assumption \ref{ass:bounded_profiles} is natural in the elliptic framework but might be contradictory in
an aggregative framework.

We believe that our results, and especially the formula \eqref{eq:sign_speed}, remain true in a weakly degenerate setting.
On the very special case where $d_{1,1}(\bullet,0)$ vanishes at $0$ and at $0$ only, we refer for instance to 
\cite{Marcelli_Papalini_2018,Sanchez-Garduno_Maini_1995,Sanchez-Garduno_Maini_1994}, and references therein. We point out that
our results would extend indeed to such a degenerate diffusivity provided a result analogous to Theorem \ref{thm:app} can
be established.

Finally, we point out that there are no sign conditions on the cross-diffusion rates $d_{1,2}$ and $d_{2,1}$.

\subsubsection{\ref{ass:no_coupling_at_zero}}
This assumption ensures that the system \eqref{eq:RD_system} is well-posed in the following sense:
if $u=0$ initially, it should not spontaneously appear, so that 
\begin{equation*}
    -\partial_x(d_{1,2}(0,v)\partial_x v)-h_{1,2}(0,v)\partial_x v=0\quad\text{for all }v.
\end{equation*}
Taking for instance $V(x)=v-x^2$, for some $v\in[0,1]$, and evaluating at the maximum $x=0$, we 
get $d_{1,2}(0,v)=0$. Subsequently, we deduce $h_{1,2}(0,v)=0$. Similarly, $d_{2,1}(u,0)=h_{2,1}(u,0)=0$ for any $u$.

Naturally, all applications we have in mind satisfy this assumption.

Note that with such an assumption, $\mathbf{D}$ and $\mathbf{H}$ are constant if and only if they are constant and diagonal.

\subsubsection{\ref{ass:monostability}}
This is a loose monostability assumption, ensuring that if $v=0$ then $u$ follows a monostable equation, and vice versa.
The instability of $0$ and the stability of $1$ are not understood here in the linear sense (hyperbolicity of the equilibrium),
but truly in the nonlinear sense (say, $\alpha$-limit set and $\omega$-limit set of the associated ODE $u'=ug_1(u,0)$). 
It is also in that sense that the function $g$ is of bistable type.

The fact that both 
carrying capacities are unitary is assumed without loss of generality, up to a nondimensionalization of both densities.

Recall that monostable reaction terms can be either of Fisher--KPP type ($g$ is maximal at $0$, \textit{e.g.} $g(z)=1-z$) 
or have a weak Allee effect ($\max_{[0,1]} g>g(0)$, \textit{e.g.} $g(z)=z(1-z)$). 

\subsubsection{\ref{ass:symmetric_competition}}
This assumption generalizes the interpopulation competition of Lotka--Volterra type, where $\omega(u,v)=uv$. In view of the
proof, it might in fact be weakened in several ways. However, we do not know any reasonable model that satisfies only 
such a weakened assumption.
Since $\omega$ does not play an important role in the paper and since we mostly have in mind $\omega(u,v)=u^p v^q$ with 
$p,q\geq 1$, we choose to keep the more explicit stronger assumption.

\subsubsection{\ref{ass:ordered_monostable_minimal_speeds}}\label{sec:ordered_monostable_minimal_speeds}
The assumption $-c_\star^-<c_\star^+$ is very natural and we actually expect it to be necessary, in the sense that
if it is not satisfied, the existence of bistable waves \ref{ass:existence_waves} should fail, as will be discussed below
in Section \ref{sec:bounded_speeds}.

This assumption basically means that, in a situation where the species $u$ comes from $x\simeq -\infty$ and the species $v$ comes from
$x\simeq +\infty$, with negligible strong couplings due to the distance between the two species, their invasion fronts
move one toward the other and will eventually meet and form a bistable traveling wave. 

Marcelli and Papalini \cite{Marcelli_Papalini_2018} recently showed that the minimal wave speeds $c_\star^+$ and $c_\star^-$ 
satisfy the following upper estimates:
\begin{equation*}
c_\star^+ \leq 2\sqrt{\sup_{z\in(0,\alpha]}\frac{1}{z}\int_{0}^z g_1\left(\frac{\hat{z}}{\alpha},0\right)d(\hat{z})\textup{d}\hat{z}} -\inf_{z\in(0,\alpha]}\frac{1}{z}\int_{0}^z h_{1,1}(\frac{\hat{z}}{\alpha},0)\textup{d}\hat{z},
\end{equation*}
\begin{equation*}
    c_\star^- \leq 2\sqrt{\sup_{z\in[-1,0)}\frac{1}{-z}\int_{z}^0 g_2(0,-\hat{z})d(\hat{z})\textup{d}\hat{z}} +\sup_{z\in[-1,0)}\frac{1}{-z}\int_{z}^0 h_{2,2}(0,-\hat{z})\textup{d}\hat{z}.
\end{equation*}
Therefore \ref{ass:ordered_monostable_minimal_speeds} fails if, for instance,
$h_{2,2}$ is a large negative constant whereas $h_{1,1}$ is a large positive constant.
Having in mind the preceding heuristics, this is natural: such an advection 
term $h$ slows down the invasion of both $u$ and $v$ so strongly that the two fronts never meet, 
with a buffer zone that is linearly increasing in time. 
Instead of a traveling wave, this describes the formation of what is nowadays called a propagating terrace \cite{Ducrot_Giletti_Matano}.
Such a phenomenon is standard in bistable dynamics and was first reported by Fife and McLeod \cite{Fife_McLeod_19}.

On the contrary, various simple conditions can ensure $-c_\star^-<c_\star^+$. 
Indeed, Malaguti and Marcelli \cite{Malaguti_Marcelli_2002} established the following lower estimates:
\begin{equation*}
    c_\star^+ \geq 2\sqrt{d_{1,1}(0)g_1(0,0)}-h_{1,1}(0,0),\quad c_\star^- \geq 2\sqrt{d_{2,2}(0)g_2(0,0)}+h_{2,2}(0,0).
\end{equation*}
For instance, $-c_\star^-<c_\star^+$ as soon as $h_{1,1}=h_{2,2}=0$ (no self-advection),
$g_1(0,0)>0$ and $g_2(0,0)>0$.
In particular, the systems \eqref{eq:LV_system}, \eqref{eq:PP_system}, \eqref{eq:SKT_system}
all satisfy \ref{ass:ordered_monostable_minimal_speeds}.

\subsubsection{\ref{ass:existence_waves}}\label{sec:existence_waves}
This is the first truly restrictive assumption. Indeed, for a system as general as \eqref{eq:RD_system},
the existence of bistable traveling waves connecting $(1,0)$ and $(0,1)$ is a completely open and difficult problem and
it might very well be false in full generality.

More precisely, although the existence of traveling waves $(u,v)(t,x)=(\phi,\psi)(x-ct)$
connecting $(1,0)$ and $(0,1)$
is to be expected under fairly reasonable assumptions, we expect that in some cases their 
profiles $\phi$ and $\psi$ will not be monotonic. In general, non-monotonic traveling waves
connecting $(1,0)$ and $(0,1)$ do not satisfy any uniform $L^\infty$ bounds or $L^1$ bounds
for the profile derivatives. Clearly, from our proof, we need such bounds. We also used
several times the fact that the profile derivatives do not change sign. In particular, the 
monotonicity of the profiles implies (almost directly) that the free boundary is a point or an 
interval. With non-monotonic profiles, the free boundary problem will be in general much more 
difficult. Even though we believe our approach can be successful when studying some specific
non-monotonic traveling waves, for which convenient bounds can be \textit{a priori} established,
the setting in the present paper is already quite abstract and we deliberately choose to
exclude non-monotonic traveling waves.

It seems to us that there are mainly three ways to solve, at least partially, this problem of 
existence of monotonic traveling waves.
\begin{enumerate}
    \item Direct construction: this is what was done in the scalar case by Malaguti, Marcelli and Matucci 
	\cite{Malaguti_Marcelli_Matucci_2004}. However
	the proof there relies upon phase-plane arguments and therefore does not extend easily to systems.
    \item Perturbative arguments starting from the semilinear case with \eqref{eq:LV_reaction}: using a standard implicit function
	theorem approach for bistable waves in monotone systems (presented in detail in, for instance, 
	\cite{Volpert_2014,Sandstede_2002}), this approach would 
	give the existence of a neighborhood of any constant
	pair $(\mathbf{D},\mathbf{H})$ in which traveling waves exist (see also \cite{Wu_Zhao_2010}
	for another type of perturbative result). However the diameter of this neighborhood 
	will depend on $k$
	and it is quite difficult to bound it from below. Therefore such an approach is hardly suitable if our ultimate goal
	is to pass to the limit $k\to+\infty$.
    \item Homotopy arguments starting from the limit $k=+\infty$: such an approach has been used repeatedly in the literature
	on strongly competitive systems (\textit{e.g.}, \cite{Girardin_Zilio,Dancer_Du_1994}). It requires first a good knowledge
	of the limit. The present paper can therefore be understood as a first step toward existence results proved with this 
	approach.
\end{enumerate}
Consequently, the proof of existence is left as a difficult but interesting open problem and, in this paper, we simply 
assume \textit{a priori} the existence.

\subsubsection{\ref{ass:bounded_speeds}}\label{sec:bounded_speeds}
This is a technical assumption that should always be satisfied as soon as traveling waves exist \ref{ass:existence_waves}. 
Indeed, it is a quite general feature of bistable waves that their speed is stuck between the minimal wave speeds
of the two underlying monostable problems (\textit{e.g.}, \cite{Fang_Zhao_2011,Malaguti_Marcelli_Matucci_2004,Gardner_1982,Kan_on_1995}). 
Here, these two bounds are precisely $c_\star^+$ and $-c_\star^-$. In other 
words, we actually expect that $-c_\star^-<c_k<c_\star^+$ is true for all $k>k^\star$.

Nevertheless, since the existence of waves is unclear and since the system does not satisfy the comparison principle (which
is the main tool used to prove the above inequalities in the semilinear competition--diffusion case \eqref{eq:LV_system}),
we have no choice but to add this likely superfluous assumption.

\subsubsection{\ref{ass:bounded_profiles}}\label{sec:bounded_profiles}
This is the second truly restrictive assumption. It is for instance satisfied if an $L^2_{\textup{loc}}$ estimate on
$\boldsymbol{\Phi}_k'$ can be proved, but it seems that in general such an estimate is false. Below we give important
examples where such an estimate can be proved indeed.

\begin{enumerate}
    \item If both $d_{1,2}$ and $d_{2,1}$ are nonpositive
	(this applies in particular to the system \eqref{eq:PP_system} with $\min(\gamma_1,\gamma_2)\geq 0$), 
	then we can multiply the first, respectively second,
	equation of the system \eqref{eq:TW_system} by $\phi_k\chi_R$, respectively $\psi_k\chi_R$, 
	where $\chi_R$ is a cut-off function equal to $1$ in $[-R,R]$, smooth and valued in $(0,1)$ in $[-R-1,-R]\cup[R,R+1]$,
	and equal to $0$ in $\mathbb{R}\backslash(-R-1,R+1)$.
	Just as in \cite{Girardin_Nadin_2015}, a few integrations by parts lead to an $L^2([-R,R])$ estimate on $\phi_k'$, respectively
	$\psi_k'$.
	For completeness, the detailed calculation for the first equation -- where subscripts
	$R$ and $k$ are dropped for ease of reading -- follows. Together, the inequality 
	$\int d_{1,2}(\boldsymbol{\Phi})\psi'\phi'\chi\geq 0$ (where $d_{1,2}\leq 0$ and $\psi'\phi'\leq 0$ are used) and the equality
	\begin{align*}
	    \int d_{1,1}(\boldsymbol{\Phi})\phi'\phi'\chi 
	    +\int d_{1,1}(\boldsymbol{\Phi})\phi'\phi\chi' &
	    +\int d_{1,2}(\boldsymbol{\Phi})\psi'\phi'\chi
	    +\int d_{1,2}(\boldsymbol{\Phi})\psi'\phi\chi' \\
	    -\int h_{1,1}(\boldsymbol{\Phi})\phi'\phi\chi
	    & -\int h_{1,2}(\boldsymbol{\Phi})\psi'\phi\chi-c\int \phi'\phi\chi \\
	    & =\int \phi^{2}g_{1}(\boldsymbol{\Phi})\chi
	    -k\int\omega(\boldsymbol{\Phi})\phi\chi 
	\end{align*}
	imply
	\begin{align*}
	    \int_{-R}^R d_{1,1}(\boldsymbol{\Phi})(\phi')^2 & \leq
	    -\int d_{1,1}(\boldsymbol{\Phi})\phi'\phi\chi'
	    -\int d_{1,2}(\boldsymbol{\Phi})\psi'\phi\chi' \\
	    &\quad +\int h_{1,1}(\boldsymbol{\Phi})\phi'\phi\chi
	    +\int h_{1,2}(\boldsymbol{\Phi})\psi'\phi\chi \\
	    &\quad +c\int \phi'\phi\chi
	    +\int \phi^{2}g_{1}(\boldsymbol{\Phi})\chi.
	\end{align*}
	It just remains to verify that each term on the right-hand side is bounded uniformly
	with respect to $k$. The last term is bounded as follows:
	\begin{align*}
	    \left|\int \phi^{2}g_{1}(\boldsymbol{\Phi})\chi\right| &
	    \leq \int\left| \phi^{2}g_{1}(\boldsymbol{\Phi})\chi\right| \\
	    & \leq \|g_{1}\|_{L^\infty([0,1]^2)}\|\phi^2\|_{L^\infty(\mathbb{R})}\|\chi\|_{L^1(\mathbb{R})} \\
	    & = \|g_{1}\|_{L^\infty([0,1]^2)}\|\chi\|_{L^1(\mathbb{R})}.
	\end{align*}
	All the other terms are handled similarly, using the facts that $(c_k)_{k>k^\star}$
	is bounded \ref{ass:bounded_speeds} and that
	$\phi'$ and $\psi'$ both have a unitary $L^1(\mathbb{R})$-norm. For instance,
	\begin{align*}
	    \left|\int d_{1,2}(\boldsymbol{\Phi})\psi'\phi\chi'\right| &
	    \leq \int\left| d_{1,2}(\boldsymbol{\Phi})\psi'\phi\chi'\right| \\
	    & \leq \|d_{1,2}\|_{L^\infty([0,1]^2)}\|\phi\|_{L^\infty(\mathbb{R})}\|\chi'\|_{L^\infty(\mathbb{R})}\|\psi'\|_{L^1(\mathbb{R})} \\
	    & = \|d_{1,2}\|_{L^\infty([0,1]^2)}\|\chi'\|_{L^\infty(\mathbb{R})}.
	\end{align*}
	In the end, using the boundedness from below of $d_{1,1}$, there exists a 
	constant $C>0$ that does not depend on $k$ or $R$ such that 
	\begin{equation*}
	    \int_{-R}^R (\phi')^2 \leq C(\|\chi\|_{W^{1,\infty}(\mathbb{R})}+\|\chi\|_{L^1(\mathbb{R})}).
	\end{equation*}
	Remark that necessarily $\|\chi\|_{W^{1,\infty}}+\|\chi\|_{L^1}$
	depends on $R$. We also point out that further integration by parts are possible but
	do not improve the estimate.
    \item If \eqref{eq:LV_reaction} and $\mathbf{D}$ is the Jacobian matrix of a $\mathscr{C}^2$-diffeomorphism 
	\[
	    \mathbf{d}:[0,+\infty)^2\to[0,+\infty)^2,
	\] 
	then we can multiply the first, respectively second, equation by $d_1(\boldsymbol{\Phi}_k)\chi_R$, respectively
	$d_2(\boldsymbol{\Phi}_k)\chi_R$, and discover similarly an $L^2([-R,R])$ estimate on $(\mathbf{d}(\boldsymbol{\Phi}_k))'$,
	that is on $\boldsymbol{\Phi}_k'$ by invertibility of $\mathbf{d}$. This idea is borrowed
	from \cite{Zhang_Zhou_Liu_2017,Zhou_Zhang_Liu_Lin_2012,Zhou_Zhang_Liu_2012}.
	As such, this applies to the system \eqref{eq:SKT_system}.
\end{enumerate}

A very important remark is required here: similar estimates cannot be obtained for the general evolution parabolic problem
or for the general diffusion--advection elliptic problem, as here we heavily use 
the monotonicity and the $L^\infty$-boundedness of the profiles for the calculations.
More details on the difficult estimates for the evolution problem can be found for instance
in \cite{Dreher_2007,Desvillettes_et_al_2015} and references therein.

Actually, the monotonicity of the profiles and the uniform $L^\infty$ bounds induced by the traveling wave
form, that are heavily used in the above calculations, directly yield some regularity,
compactness and convergence properties. Unfortunately, these are not sufficient to derive the 
limiting problem and \ref{ass:bounded_profiles} seems to be truly required. Still, for possible future reference,
we list some of these properties below.
\begin{enumerate}
    \item Since the wave profiles are all of bounded variation and uniformly bounded, we can directly apply Helly's selection
	theorem: up to extraction, $(\boldsymbol{\Phi}_k)_{k>k^\star}$ converges pointwise and locally in $L^1$ to
	some limit $\boldsymbol{\Phi}_\infty$. This does not depend on the equations.
    \item The families $(-\phi_k'\textup{d}\xi)_{k>k^\star}$ and $(\psi_k'\textup{d}\xi)_{k>k^\star}$ are families of 
	probability measures in $\mathbb{R}$.
	By the Banach--Alaoglu theorem, a bounded family of Radon measures in $\mathbb{R}$ is relatively 
	compact in the weak-$\star$ topology of the set of bounded Radon measures, 
	namely the topology of pointwise convergence on the space of continuous functions in $\mathbb{R}$ that converge 
	to $0$ at $\pm\infty$. Again, this does not depend on the equations.
    \item A monotone function is differentiable almost everywhere, with a number of discontinuities at most countable,
	and each discontinuity is a jump discontinuity. 
\end{enumerate}

\subsection{On the results}
\subsubsection{The limiting equation}

The limiting equation in \eqref{eq:singular_limit} does not have classical solutions (namely, solutions of class $\mathscr{C}^2$) 
in general. In fact, the unique weak solution $z$ is of class $\mathscr{C}^1$ if and only if the left-sided and right-sided 
limits of $d$ at $0$ coincide (that is $d_{1,1}(0,0)=d_{2,2}(0,0)$)
and of class $\mathscr{C}^2$ if and only if it is of class $\mathscr{C}^1$ and the left-sided and right-sided limits of $h$ at $0$ coincide
(that is $h_{1,1}(0,0)=h_{2,2}(0,0)$). Nevertheless, since the equation implies that 
$d(z)z'\in W^{1,\infty}(\mathbb{R})$, $d(z)z'$ is at least Lipschitz-continuous. This is where the free boundary relation comes from.

As they are defined in \eqref{def:d} and \eqref{def:h}, the functions $d$ and $h$ both have 
a zero at $z=0$.
However, Theorem \ref{thm:main} does not depend on how the piecewise-continuous functions $d$ and $h$ are defined at $z=0$. 
This is classical in such problems (\textit{e.g.}, \cite{Dancer_Hilhors,Crooks_Dancer__1,Girardin_Nadin_2015}).

Note that the positivity of the left-sided and right-sided limits of $d$ at $0$, that is the positivity of
the essential infimum of $d$, still matters.

\subsubsection{The convergence}
The space where the convergence of the profiles occurs might be improved by bootstrap in 
special cases, using more deeply the structure of the equations. 
We point out that, in view of the regularity of the limit, the best that can be proved is a uniform Lipschitz
bound (implying convergence in all H\"{o}lder spaces $\mathscr{C}^{0,\gamma}$) for $(\boldsymbol{\Phi}_k)_k$ and 
$(\mathbf{D}(\boldsymbol{\Phi}_k)\boldsymbol{\Phi}_k')_k$. This is consistent with the literature on the semilinear problem
(refer to Soave--Zilio \cite{Soave_Zilio_2015} for a detailed review).

\subsubsection{The sign of the effective wave speed when the self-advection is constant}

The formula \eqref{eq:sign_speed}
is the extension of the ``Disunity is strength''-type result for the semilinear competition--diffusion system \eqref{eq:LV_system}
established in \cite{Girardin_Nadin_2015}.

It shows for instance that, if $h_0=0$, $g_1(\bullet,0)=g_2(0,\bullet)$, $\alpha=1$ and $d_{1,1}(\bullet,0) < d_{2,2}(0,\bullet)$,
so that the two populations differ only in diffusivity and the diffusivities obtained in absence of the other are strictly ordered,
then $c_\infty<0$: again, a strong diffusivity is a competitive advantage.

Quite interestingly, the formula \eqref{eq:sign_speed} does not depend on the cross-diffusion rates $d_{1,2}$ and $d_{2,1}$, on the
cross-advection rates $h_{1,2}$ and $h_{2,1}$ and on the competition rate $\omega$. In fact, it only
depends on the parameter $\alpha$ and on the dynamics of each population in absence of the other: although 
the system is arbitrarily strongly coupled, its spatial segregation limit is mostly decoupled (more 
precisely, it is coupled only at the free boundary).

In order to illustrate more directly our result, below, we apply the formula \eqref{eq:sign_speed} to the specific cases studied 
in the aforementioned earlier literature \cite[etc.]{Potts_Petrovskii,Desvillettes_et_al_2015}. This mostly amounts to
comparing dispersal strategies, but clearly we can also use the formula \eqref{eq:sign_speed} to compare growth strategies
and determine, for instance, the effect of a weak Allee effect.

First, we consider self- and cross-diffusion systems with logistic growth, namely systems
generalizing \eqref{eq:SKT_system}, where
\eqref{eq:LV_reaction} holds, $\mathbf{H}=\mathbf{0}$ and $\mathbf{D}(u,v)$ is the following Jacobian matrix:
\begin{align*}
    \mathbf{D}(u,v) & =\textup{Jac}
    \begin{pmatrix}
	u(d_1+a_{1,1}u^{\beta_{1,1}}+a_{1,2}v^{\beta_{1,2}}) \\
	v(d_2+a_{2,1}u^{\beta_{2,1}}+a_{2,2}v^{\beta_{2,2}})	
    \end{pmatrix} \\
    & =
    \begin{pmatrix}
	d_1+a_{1,1}(\beta_{1,1}+1)u^{\beta_{1,1}}+a_{1,2}v^{\beta_{1,2}} & a_{1,2}\beta_{1,2}uv^{\beta_{1,2}-1} \\
	a_{2,1}\beta_{2,1}vu^{\beta_{2,1}-1} & d_2+a_{2,2}(\beta_{2,2}+1)v^{\beta_{2,2}}+a_{2,1}u^{\beta_{2,1}} 
    \end{pmatrix}.
\end{align*}

In such a case, the formula \eqref{eq:sign_speed} reads
\begin{equation*}
    \operatorname{sign}\left(c_\infty\right)=\operatorname{sign}\left(\alpha^2\left(\frac{d_{1}}{6}+\frac{a_{1,1}(\beta_{1,1}+1)}{\beta_{1,1}^2+5\beta_{1,1}+6}\right)-r\left( \frac{d_{2}}{6}+\frac{a_{2,2}(\beta_{2,2}+1)}{\beta_{2,2}^2+5\beta_{2,2}+6} \right)\right).
\end{equation*}

In the special case $\beta_{1,1}=\beta_{2,2}=1$, this reads
\begin{equation*}
    \operatorname{sign}\left(c_\infty\right)=\operatorname{sign}\left(\alpha^2\left(d_{1}+a_{1,1}\right)-r\left( d_{2}+a_{2,2} \right)\right).
\end{equation*}

Hence the sign of the wave speed is determined by the self-diffusivities in absence of the other and at carrying capacity. This
illustrates quite interestingly the fact that the coupling at the free boundary is not just local but takes also into account what
happens far away from the free boundary. The above formula also shows that there are cases where a population wins despite
a smaller linear diffusivity (say, for instance, $\alpha=r=1$, $d_1+a_{1,1}>d_2+a_{2,2}$ and $d_1<d_2$).

Note that we can apply formally the formula to the case $d_1=0$ or $d_2=0$ (no linear diffusion, just porous-medium type self-diffusion). 
Although our proof does not apply in such a case (see \ref{ass:ellipticity}), we conjecture that this formula is indeed the correct one.

Second, we consider cross-taxis systems with logistic growth, namely systems of the form 
\eqref{eq:PP_system}, where \eqref{eq:LV_reaction} holds, $\mathbf{H}=\mathbf{0}$ and
\begin{equation*}
    \mathbf{D}(u,v)=
    \begin{pmatrix}
	d_1 & -\gamma_1 u \\
	-\gamma_2 v & d_2
    \end{pmatrix}.
\end{equation*}

It turns out that for such a system, we recover exactly the same formula as in the case of constant $\mathbf{D}$:
the sign of $c_\infty$ is that of $\alpha^2 d_1-rd_2$. The values of $\gamma_1$ and $\gamma_2$ have no effect. Hence
``Disunity is strength'' again, whatever the strength or the sense (attraction or repulsion) of the cross-taxis.

This conclusion is in sharp contrast with the ``Fortune favors the bold'' conclusion suggested numerically by Potts and Petrovskii in 
\cite{Potts_Petrovskii}. The contrast is not contradictory and is explained by the fact that their simulations 
take finite values of $k$ whereas we focus on the limit $k\to+\infty$, where the effect of the aggressive taxis
becomes negligible due to spatial segregation.

Let us end this discussion by relating this work to the wide literature on mathematical models for
the evolution of dispersal in nature (we refer for instance to \cite[Chapter
11]{Cantrell_Cosner_Ruan_2009} and references therein).  Adopting such a viewpoint, our result tends
to show that, for a species trying to survive a strong competitor, the formation by means of
mutations of an aggressive phenotype -- in the cross-taxis sense of Potts and Petrovskii --, or more
generally of a phenotype whose dispersal strategy accounts for the competitor, should actually not
help, and therefore such a mutant should not be selected. This might seem counter-intuitive at
first, yet it can be explained heuristically at the individual scale: the infinite competition limit
means that as soon as two individuals of adversary populations meet, both die, instantaneously.
Since no one will ever survive such an encounter, even for the tiniest amount of time, it is
therefore useless to try to learn how to deal with these encounters. As a matter of fact, according
to the model, the fittest mutations in such a situation are those that improve the diffusion of the
species, be it linear diffusion or nonlinear self-diffusion. This is of course in sharp contrast
with the well-known ``Unity is strength''-type result in spatially heterogeneous environments
without strong competitor proved by Dockery \textit{et al.} \cite{Dockery_1998} and this might lead
to evolutionary traps, as explained by the first author in \cite[Section 3.5]{Girardin_2019}.

\appendix
\section{The fully nonlinear bistable scalar equation with piecewise-continuous diffusion and advection rates}\label{app}
In this appendix, we sketch briefly the proof of the following theorem that confirms that the functions $d$ and $h$
can be defined arbitrarily at $0$. This theorem extends the main result of Malaguti--Marcelli--Matucci 
\cite{Malaguti_Marcelli_Matucci_2004}.

\begin{thm}\label{thm:app}
    Let $U_\pm\in\mathbb{R}$ such that $U_-<0<U_+$ and let
    \begin{enumerate}
	\item $d:[U_-,U_+]\to\mathbb{R}$ be continuously differentiable and positive in $[U_-,U_+]\backslash\left\{ 0 \right\}$ with positive limits at $0^\pm$;
	\item $h:[U_-,U_+]\to\mathbb{R}$ be continuous in $[U_-,U_+]\backslash\left\{ 0 \right\}$ with limits at $0^\pm$;
	\item $g:[U_-,U_+]\to\mathbb{R}$ be continuous with $g(U_-)=g(0)=g(U_+)=0$ and $g<0$ in $(U_-,0)$ and $g>0$ in $(0,U_+)$.
    \end{enumerate}

    Let $c_\star^\pm$ be the minimal wave speed of nonincreasing classical solutions $(z,c)$ of 
    \begin{equation}
	-(d(z)z')'-(c+h(z))z'=g(z)
	\label{eq:app}
    \end{equation}
    with limits $0$ at $+\infty$ and $U_\pm$ at $-\infty$. 

    Then, if $-c_\star^-<c_\star^+$, the equation \eqref{eq:app} supplemented with the constraints
    \begin{enumerate}
	\item $z$ admits $U_+$ and $U_-$ as limits at $-\infty$ and $+\infty$ respectively;
	\item $z$ is nonincreasing;
	\item $z$ is continuous;
    \end{enumerate}
    admits a unique (up to shifts) weak solution $(z,c)$, and its wave speed $c$ satisfies $-c_\star^-<c<c_\star^+$.

    Conversely, if $-c_\star^-\geq c_\star^+$, no such solution exists.
\end{thm}

\begin{proof}
    The existence of the wave, its monotonicity, its continuity and its limits when $-c_\star^-<c_\star^+$ 
    directly follow from a standard regularization of the functions $d$ and $h$, from a locally uniform $L^2$
    estimate on the regularized profile derivative (see \cite{Girardin_Nadin_2015} or Section \ref{sec:bounded_profiles}
    above) and from the continuity with
    respect to $d$ and $h$ of the minimal wave speed for monostable equations \cite{Malaguti_Marcelli_Matucci_2010}. 
    When passing to the limit,
    the nontriviality of the limit $z$ has to be verified, as in the main proof of the paper. In order to do so, a 
    normalization $z(0)=\frac12 U_\pm$ depending on the limiting wave speed can be used, as in the main proof of this paper. 
    Note that at this point, we have $-c_\star^-\leq c\leq c_\star^+$ but not the strict inequalities.

    In the sequel, we will need the continuity of $d(z)z'$. Although it is not assumed \textit{a priori}, 
    it can be established by integrations by parts, exactly as in the main proof of this paper. In fact,
    our proof actually shows that any monotonic continuous traveling wave weak solution necessarily satisfies $d(z)z'\in \mathscr{C}(\mathbb{R})$.

    The uniqueness (up to shifts) can be established by contradiction, assuming that there are two different nonincreasing
    continuous waves $(z_1,c_1)$ and $(z_2,c_2)$ with, for instance, $c_1\leq c_2$. Then we have two
    different positive semi-waves $(z_1^+,c_1)$ and $(z_2^+,c_2)$ or two different negative semi-waves $(-z_1^-,c_1)$ and
    $(-z_2^-,c_2)$. The positive semi-waves are solutions of
    \begin{equation*}
	\begin{cases}
	    -(d(z)z')'-(c+h(z))z'=g(z)\quad\text{in }(-\infty,0) \\
	    z(0)=0 \\
	    z>0\quad\text{in }(-\infty,0) \\
	    \lim_{-\infty}z=U_+
	\end{cases}
    \end{equation*}
    and the negative semi-waves are solutions of
    \begin{equation*}
	\begin{cases}
	    -(d(z)z')'-(c+h(z))z'=g(z)\quad\text{in }(0,+\infty) \\
	    z(0)=0 \\
	    z<0\quad\text{in }(0,+\infty) \\
	    \lim_{+\infty}z=U_-
	\end{cases}
    \end{equation*}
    For both sub-problems, the functions $d$ and $h$ are continuous, so that the uniqueness result 
    for semi-waves \cite[Lemma 10]{Malaguti_Marcelli_Matucci_2004} implies directly that $c_1<c_2$, and then
    the comparison principle \cite[Lemma 10]{Malaguti_Marcelli_Matucci_2004} implies 
    \begin{equation*}
	0>\lim_{z\to 0^+}d(z)\lim_{\xi\to 0^-}z_1'(\xi)>\lim_{z\to 0^+}d(z)\lim_{\xi\to 0^-}z_2'(\xi),
    \end{equation*}
    \begin{equation*}
	0>\lim_{z\to 0^-}d(z)\lim_{\xi\to 0^+}z_2'(\xi)>\lim_{z\to 0^-}d(z)\lim_{\xi\to 0^+}z_1'(\xi),
    \end{equation*}
    so that the continuity of either $d(z_1)z_1'$ or $d(z_2)z_2'$ is contradicted:
    \begin{equation*}
	\frac{\lim_{z\to 0^+}d(z)}{\lim_{z\to 0^-}d(z)}=\frac{\lim_{\xi\to 0^+}z_1'(\xi)}{\lim_{\xi\to 0^-}z_1'(\xi)}
	>\frac{\lim_{\xi\to 0^+}z_2'(\xi)}{\lim_{\xi\to 0^-}z_2'(\xi)}=\frac{\lim_{z\to 0^+}d(z)}{\lim_{z\to 0^-}d(z)}.
    \end{equation*}

    The strict inequalities $-c_\star^-<c<c_\star^+$ are obtained with the same argument (the uniqueness result
    \cite[Lemma 10]{Malaguti_Marcelli_Matucci_2004} does not distinguish semi-waves and regular traveling waves, so that the
    two cannot coexist).

    The nonexistence when $-c_\star^-\geq c_\star^+$ follows similarly from \cite[Lemma 10]{Malaguti_Marcelli_Matucci_2004}.
\end{proof}

\bibliographystyle{plain}
\bibliography{ref.bib}

\end{document}